\documentclass[12pt,psamsfonts]{amsart}
\usepackage{amsmath,amsthm,amsfonts,amssymb}
\usepackage{eucal,amscd}
\usepackage{graphicx}
\usepackage[all,knot]{xy}
\xyoption{arc}
\usepackage{pstricks}
\usepackage{bm}
\usepackage{epstopdf}
\usepackage{color}
\usepackage[foot]{amsaddr}

\newcommand{\rmsize}{\rm size}
\newcommand{\pow}{d}

\newcommand{\cP}{\mathcal{P}}
\newcommand{\cC}{\mathcal{C}}
\newcommand{\bF}{\mathbb{F}}
\newcommand{\F}{\mathbb F}
\newcommand{\B}{\mathcal{B}}
\newcommand{\Z}{\mathbb Z}
\newcommand{\C}{\mathbb C}
\newcommand{\e}{{\epsilon}}
\newcommand{\NOTi}{{\rm{NOT}}^{-,+}_i}
\newcommand{\NOT}{{\rm{NOT}}^{-,+}}
\newcommand{\g}{\mathfrak{g}}
\newcommand{\ch}{\check{h}}
\newcommand{\one}{\mathbf{1}}
\newcommand{\cosm}{\rm{cos}(\frac{\pi}{m})}
\newcommand{\sinm}{\rm{sin}(\frac{\pi}{m})}
\newcommand{\xornot}{\Lambda^2_{\rm{XOR}}\rm{NOT}}
\newcommand{\cI}{\mathcal{I}}
\newcommand{\cA}{\mathcal{A}}

\newcommand{\bee}{\begin{equation}}
\newcommand{\eee}{\end{equation}}

\DeclareMathOperator{\SU}{SU}
\newcommand{\su}{\mathfrak{su}}
\DeclareMathOperator{\SO}{SO}
\newcommand{\so}{\mathfrak{so}}
\DeclareMathOperator{\Sp}{Sp}
\DeclareMathOperator{\PSp}{PSp}
\DeclareMathOperator{\Irr}{Irr}
\DeclareMathOperator{\Hom}{Hom}
\DeclareMathOperator{\im}{i}
\DeclareMathOperator{\End}{End}

\newcommand{\om}{\omega}

\DeclareMathOperator{\U}{U}

\newcommand{\ot}{\otimes}

\newcommand{\bfe}{\mathbf{e}}

\newcommand{\be}{\begin{equation}}
\newcommand{\ee}{\end{equation}}

\numberwithin{equation}{section}
\newtheorem{thm}{Theorem}

\newtheorem{defn}[equation]{Definition}

\newtheorem{lemma}[equation]{Lemma}

\newtheorem{conj}[equation]{Conjecture}

\newtheorem{prop}[equation]{Proposition}

\theoremstyle{definition}

\newtheorem{definition}[equation]{Definition}

\begin{document}

\title[Metaplectic UMCs]
{On Metaplectic Modular Categories and their applications}

\author{Matthew B. Hastings$^{1}$, Chetan Nayak$^{1,2}$, and Zhenghan Wang$^1$}
\address{$^1$Microsoft Station Q\\ University of California\\ Santa Barbara, CA 93106}
\address{$^2$Department of Physics\\University of California\\Santa Barbara, CA 93106}


\thanks{M.H was partially supported by a Simons Investigator award from the Simons Foundation.  C.N. is partially supported by the DARPA QuEST program and the AFOSR under grant FA9550-10-1-0524.  Z.W. is partially supported by NSF DMS 1108736.  We thank I. Arad for pointing out reference \cite{GoldbergJerrum}}

\begin{abstract}

For non-abelian simple objects in a unitary modular category, the density of their braid group representations, the $\#P$-hard evaluation of their associated link invariants, and the BQP-completeness of their anyonic quantum computing models are closely related.  We systematically study such properties of the non-abelian simple objects in the metaplectic modular categories $SO(m)_2$ for an odd integer $m\geq 3$.  The simple objects with quantum dimensions $\sqrt{m}$ have finite image braid group representations, and their link invariants are classically efficient to evaluate.  We also provide classically efficient simulation of their braid group representations.  These simulations of the braid group representations can be regarded as qudit generalizations of the Knill-Gottesmann theorem for the qubit case.   The simple objects of dimension $2$ give us a surprising result:  while their braid group representations have finite images and are efficiently simulable classically after a generalized localization, their link invariants are $\#P$-hard to evaluate exactly.  We sharpen the $\#P$-hardness by showing that any sufficiently accurate approximation of their associated link invariants is already $\#P$-hard.

\end{abstract}

\maketitle

\section{Introduction}

Unitary modular tensor categories (UMCs) are intricate algebraic structures arose in a variety of fields, in particular in the study of topological quantum field theories (TQFTs) \cite{turaev}, conformal field theories (CFTs) \cite{MS}, and topological phases of matter \cite{Wang}.  Mathematical constructions of UMCs include representation theories of quantum groups at roots of unity and vertex operator algebras.  In condensed matter physics, each UMC is a theoretical anyonic system, and a simple object models an anyon.  Therefore, we will use the words {\it simple object} in a UMC and {\it anyon} interchangeably.    In topological quantum computation, any non-abelian anyon can be used to construct an anyonic quantum computing model \cite{Wang}.  The realization of UMCs in real physical systems and the universality of anyonic quantum computing models inspire many new mathematical problems.

In this context, the most studied sequence of UMCs is $\SU(2)_k, k=1,2,\cdots$.  One reason is the conjectured relations between $\SU(2)_2, \SU(2)_4$, and $\SU(2)_3$ and the potential non-abelian statistics in fractional quantum Hall states with filling fractions $\nu=\frac{5}{2}, \frac{8}{3}$, and $\frac{12}{5}$ \cite{RR}.  By the level-rank duality, this sequence is essentially $\SU(k)_2$---the $A$-series at level two.  In this paper, we will focus on the $B$-series at level two---$\SO(m)_2$ for an odd integer $m\geq 3$.  (The level-rank duality for $SO(m)_2$ is more involved, in particular $SO(m)_2$ is not the same as $SO(2)_m$).  We will call any UMC with the same fusion rules of $SO(m)_2$ for some odd integer $m\geq 3$ a {\it metaplectic} modular category.  Our choice of {\it metaplectic} is motivated by the connection of such UMCs and the metaplectic representations of the finite symplectic groups $Sp(2n, F_m)$ when $m$ is a prime $p$.

There are several reasons to be interested in metaplectic modular categories.  One is their conjectured property $F$ and related explicit locality, while the other is the solutions that they provided for the generalized Yang-Baxter equations \cite{RW,KW}.  But our interest mainly comes from their potential realization in condensed matter systems. With this application in mind, we study problems inspired by their application to quantum computing, in particular analyzing their potential for topological quantum computing.  In the companion paper \cite{HNW}, we addressed their relevance to some proposed physical systems.

Let $X$ be a simple object in a UMC $\cC$.  We call $X$ {\it non-abelian} if its quantum dimension $d_X>1$ because $X$ will model a non-abelian anyon in physics; otherwise, $X$ is abelian, and then $d_X=1$.  We are mainly interested in non-abelian simple objects.
Associated to each simple object $X$ in a UMC $\cC$ is a unitary braid group representation $\rho_X$ and an isotopy invariant $\cI_X(L)$ of framed links $L$ \cite{turaev}.  Modulo subtleties of encoding, the representation matrices of $\rho_X$ can be used as quantum circuits for an anyonic quantum computing model $\cA_X$ \cite{Wang}.  We are interested in the BQP-completeness of such anyonic quantum computing models.  Then the density of the  braid group representations $\rho_X$, the $\#P$-hardness of evaluating $\cI_X(L)$, and the BQP-completeness of $\cA_X$ are all closely related.

For the non-abelian simple objects $X$ of quantum dimensions $d_X=2cos(\frac{\pi}{k+2})$ in $\SU(2)_k$, the density, the $\#P$-hardness, and the BQP-completeness all match perfectly \cite{FLW1, FLW2}.  Non-abelian simple objects with all three properties will be called {\it strong}, and otherwise called {\it weak}.  Note that abelian simple objects are always weak in this sense.  The weakest non-abelian object will have finite image braid group representation, polynomial time algorithm for the evaluation of the associated link invariant, and a classically efficiently simulable computing model $\cA_X$ possibly after a localization in the sense in \cite{RW}.  In the sequence $\SU(2)_k$, non-abelian simple objects of quantum dimensions $d_X=2cos(\frac{\pi}{k+2})$ are weak if and only if $k=2,4$.  Note that the case $k=1$ is abelian.  When $k=2$, this is the Ising anyon $\sigma$, the finiteness of the braid group representation, polynomial time computable link invariant, and classically simulable computing model are all known \cite{jones83,jones87}.  So the Ising anyon is weak in all three aspects, where the classical simulation of the model $\cA_{\sigma}$ follows from the Knill-Gottesmann simulation of Clifford circuits (See e.g. \cite{NC}).  For level $k=4$, the finiteness of braid images and polynomial computable invariants are also known \cite{jones87}.  In this paper, we show that efficient classical simulation of braidings is also possible after a localization.  Therefore, a natural question is if it is possible for a non-abelian anyon to be weak in only one or two aspects.

We systematically study non-abelian simple objects in the metaplectic modular categories $SO(m)_2$.  We prove that the simple objects with quantum dimensions $\sqrt{m}$ in $SO(m)_2$ are weak in all three aspects, but the simple objects of quantum dimension $2$ in $JK_{r=6}$ give us a surprising result:  while their braid group representations have finite images and are efficiently classically simulable after a generalized localization, their link invariants are $\#P$-hard to evaluate exactly.  Actually we prove that any sufficiently accurate approximation of their associated link invariants is $\#P$-hard.  For strong anyons, similar hardness result for approximation was obtained in \cite{kuperberg}.

The contents of the paper are as follows.  In Section $2$, the basic data for metaplectic modular categories are listed.  In Section $3$, we analyze  the braid group representations associated to the quantum dimension $2$ simple objects.  It follows from our analysis that their braid group representations have finite images.  In Section $4$, we provide classical efficient simulation of the Gaussian braid representations and the braid representations from \cite{KW}.  These simulations can be regarded as qudit generalizations of the Knill-Gottesmann theorem.  In Section $5$, we prove that any sufficiently accurate approximation of the link invariant associated with the quantum dimension $2$ simple objects is already $\#P$-hard.

\section{Metaplectic Modular Categories}

One systematic way to construct UMCs is via the representation theory of quantum groups at a particular root of unity \cite{turaev}.  For each simple Lie algebra $\g$ and an integer $k\geq 1$, called the level of the theory,  two UMCs can be constructed from the representation theory of $U_q(\g)$, where $q=e^{\pm \frac{\pi \im}{l}}, l=m_{\g}(k+\ch_{\g})$, where $\ch_{\g}$ is the dual Coxeter number of $\g$ and $m_{\g}=1$ for $A,D,E$, $2$ for $B,C, F_4$, and $3$ for $G_2$.  For $\SU(2)$,  $\ch_{\su(2)}=2$ and for $\SO(m)$, $\ch_{\so(m)}=m-2$.  The two UMCs for $q=e^{\pm \frac{\pi \im}{l}}$ are conjugates of each other, so for definiteness $G_k$ denotes the category from $q=e^{\frac{\pi \im}{l}}$, where $G$ is the simple Lie group with Lie algebra $\g$.

\subsection{$\SO(m)_2$ UMCs}

The UMC $\SO(m)_2$ is of rank $r+4$, where $m=2r+1, r\geq 1$.   Their simple object representatives will be denoted as $\Irr\{\SO(m)_2\}=\{\one, Z, X_{\epsilon}, X'_{\epsilon}, Y_j, 1\leq j \leq r\}$.  In \cite{NR}, $\one$ is denoted as $X_0$, $Z$ as $X_{2\lambda_1}$, and  $Y_j$ as $X_{\gamma^j}$.
Their quantum dimensions are $d_\one=d_Z=1, d_{X_{\epsilon}}=d_{X'_{\epsilon}}=\sqrt{m}, d_{Y_j}=2, 1\leq j \leq r$.

All fusion rules of those UMCs can be deduced from the ones listed below.  In the following, we also denote $\one$ by $Y_0$ sometimes.

\begin{enumerate}

\item $X_\e\otimes X_\e\cong \oplus_{j=0}^{j=r}Y_{j}$

\item $X_\e\otimes Y_{j}\cong X_\e\oplus X_\e', 1\leq j \leq r$

\item $X_\e\otimes X_\e'\cong  Z\oplus \oplus_{j=1}^{j=r}Y_{j}$

\item $Z \otimes X_\e=X_\e'$

\item $Z \otimes Z=\one$

\item $Z \otimes Y_{j}=Y_j, 1\leq j \leq r$

\item $Y_j\otimes Y_j=\one \oplus Z\oplus Y_{\textrm{min}\{2j, m-2j\}}, 1\leq j \leq r$

\item $Y_i\otimes Y_j=Y_{|i-j|}\oplus Y_{\textrm{min}\{i+j, m-i-j\}}, 1\leq i,j \leq r, i\neq j$.

\end{enumerate}


 Note that the fusion rules for the sub-category consisting of $\{\one,Z, Y_j, 1\leq j\leq r\}$ are exactly the same as those of the representations of the dihedral group $D_m$ of order $2m$.

The twist or topological spin of a simple object $X$ is of the form $\theta_X=e^{2\pi \im h_X}$, where the rational number $h_X$, only defined modulo $1$, is called the scaling dimension of $X$.  For $SO(m)_2$, we have:

$$\; h_Z=1,\; h_{X_\e}=\frac{r}{8},\; h_{X_\e'}=\frac{r+4}{8},\; h_{Y_j}=\frac{j(m-j)}{2m}, 0\leq j\leq r.$$

Some braidings for $\SO(m)_2$ are as follow:

\begin{equation}\label{data}
\begin{array}{c} R^{Y_1,Y_1}_{\mathbf{1}}=e^{\pi \im (m+1)/m}, R^{Y_1,Y_1}_{Z}=e^{\pi \im /m}, R^{Y_1,Y_1}_{Y_2}=e^{\pi \im (m-1)/m}\\
R^{X_\e, X_\e}_{Y_j}={\im}^{(r-j)(r-j+1)-j} e^{\pi \im (\frac{r}{4} +\frac{j^2}{4r+2})}, j=0, 1,2,..,r.
\end{array}
\end{equation}

\subsection{Metaplectic modular categories}

\begin{defn}

Let $\cC$ be a UMC, and $X$ be a simple object.

\begin{enumerate}

\item $\cC$ is metaplectic if its fusion rules are the same as those of $SO(m)_2$ for some odd integer $m\geq 3$.

\item $X$ is called non-abelian if its quantum dimension $d_X>1$; otherwise, $X$ is called abelian, and then $d_X=1$

\item An abelian simple object is called a boson if its twist $\theta_X=1$, a fermion if $\theta_X=-1$, and a semion if $\theta_X=\pm \im$.

\end{enumerate}

\end{defn}

We will use the same notation for simple objects in $SO(m)_2$ as for all metaplectic modular categories.
The fusion of an abelian simple object with other simple objects only permutates them.  From the above fusion rules, we see that the simple object $Z$ of any metaplectic modular category is a boson.  A complete classification of all metaplectic modular categories is open.  Work towards a classification is obtained in \cite{TWen,RWen}.

When $m=3$, metaplectic modular categories are of rank=$5$.  In \cite{BNRW}, all rank=$5$ UMCs are classified.  In particular, there are $4$ metaplectic modular categories for $m=3$:  the $SU(2)_4$ and its conjugate, the Jones-Kauffman theory at $A=ie^{-\frac{2\pi i}{24}}$, denoted as $JK_{r=6}$, and its conjugate. The $SU(2)_4$ and the $JK_{r=6}$ are distinguished by the Frobenius-Schur indicator of $X_\e$: $-1$ for $SU(2)_4$ and $1$ for $JK_{r=6}$.

The full data of $SU(2)_4$ and $JK_{r=6}$ can be found in \cite{Bonderson} and \cite{KauffmanLins}, respectively.  While the $F$-matrices in \cite{Bonderson} are unitary, the $F$-matrices in \cite{KauffmanLins} need to be renormalized as in \cite{Wang} to become unitary.

Each simple object in a metaplectic modular category has an integral squared quantum dimension.  For braid group representations, the afforded representations by $X_\e'$ and $X_\e$ differ only by a one-dimension representation, so we will only discuss the one from $X_\e$.  Also due to the similarities among $Y_i$'s, we will focus only on $Y_1$.
The equivalence of the braid representation of $X_\e$ in $SO(m)_2$ and the Gaussian representation below is proven in \cite{RWen}.

\section{Extra-special $p$-groups and metaplectic representations}

\subsection{Braid group representations}

Suppose $\cC$ is a UMC, and $\Irr\{\cC\}=\{X_i\}_{i\in I}$ is a set of simple object representatives, i.e. one from each isomorphism class of simple objects.  Then each simple object $X$ in $\cC$ gives rise to a unitary representation $\rho_X=\{\rho_{X,n}\}$ of the braid group $\B=\{\B_n\}$.  Such representations for the $n$-strand braid group $\B_n$ can be conveniently described using an algebra: $A_{n,X}=\Hom(X^{\ot n},X^{\ot n})$.  The algebra $A_{n,X}$ is an analogue of a regular representation for a finite group, and $A_{n,X}\cong \oplus_{i\in I}\End(V_{X^{\ot n},X_i})$, where $V_{X^{\ot n},X_i}=\Hom(X^{\ot n},X_i)$.  The representation $V_{X^{\ot n},X_i}$ of $\B_n$, denoted also as $\rho_{X,n}$, can be conveniently described using graphical calculus.  An orthonormal basis of $V_{X^{\ot n},X_i}$ is given by admissible labelings of any connected uni-trivalent tree \cite{Wang}.

In \cite{RW, GHR}, localization and generalized localization are studied for a simple object $X$ whose squared quantum dimension is an integer.  A (generalized) localization of the braid group representation $\rho_X$ is a single unitary (generalized) $R$-matrix $R_X$ such that the natural braid group representation from $R_X$, denoted as $\rho_{R_X}$, is equivalent (as a representation) to a braid representation $\tilde{\rho}_X$ constructed from $\rho_X$ as follows: decomposing each representation $\rho_{X,n}$ of the $n$-strand braid group $\B_n$ into its irreducible summands:  $\rho_{X,n}=\oplus_{j}\rho_{X,n,j}$, then $\tilde{\rho}_{X,n}=\oplus_{j}m_{X,n,j} \rho_{X,n,j}$ for some non-zero multiplicities $m_{X,n,j}$.

The localizing qudit representation $\rho_{R_X}$ has a natural tensor product structure, similar to the quantum circuit model.  Therefore, we will refer to a localizing representation $\rho_{R_X}$ of a simple object $X$ as its {\it qudit representation} associated to $R_X$.

For each such object $X$ localized by a (generalized) $R$-matrix $R_X$, we have two braid group representations $\rho_X$ and $\rho_{R_X}$, which have the same irreducible summands.  The question that we are interested in is:

\vspace{.1in}

{\it For each $n$, is the image $\rho_X(\B_n)$, or equivalently $\rho_{R_X}(\B_n)$,  a finite group?  If so, which finite group?}

\subsection{Finiteness for the representation $\rho_{X_\e}$}

\subsubsection{Extra-special $p$-groups}

Let $X_i$ (or $Z_i$) be the Pauli matrices $\sigma_X$ (or $\sigma_Z$) acting on the $i$-th qubit of the $n$-qubits $(\C^2)^{\ot n}$.
The {\it real} Pauli group $\cP_n$ on $n$-qubits is the group generated by all such Pauli gates.  (The complex Pauli groups are slightly more complicated.)  As an abstract group, it is an extra-special $2$-group isomorphic to the central product of $n$-copies of the dihedral group $D_4$ of order $8$.  Its outer automorphism group $\textrm{Out}(\cP_n)\cong {\textrm{Aut}(\cP_n)}/{\textrm{Inn}(\cP_n)}$ is an orthogonal group $O^{\pm}(2n,\F_2)$ over the field $\F_2$ \cite{Griess}.  The definition of the orthogonal groups $O^{\pm}(2n,\F_2)$ does not use the usual bilinear forms, which would lead to the symplectic group $\Sp(2n,\F_2)$.  Instead they are defined using quadratic forms over the field $\F_2$.  Over the $\F_2$-vector space $\F_2^{2n}$, there are two inequivalent non-degenerate quadratic forms with maximal isotropic subspaces of dimensions $n$ and $n-1$, respectively.  The orthogonal groups $O^{\pm}(2n,\F_2)$ are the linear transformations of $\F_2^{2n}$ that preserve these two quadratic forms, respectively.  For the Pauli group $\cP_n$, the quotient ${\cP_n}/{Z(\cP_n)}$ is $\F_2^{2n}$ with a natural quadratic form $q$: $q(v)=v^2\in Z(\cP_n)\cong \F_2, v\in \F_2^{2n}$.

The Clifford group $\cC_n$ on $n$-qubits is the normalizers of the Pauli group $\cP_n$ in the unitary group $\U(2^n)$ up to sign.  The two groups fit into the following short exact sequence:

$$1\rightarrow \cP_n \rightarrow \cC_n \rightarrow \Sp(2n,\F_2) \rightarrow 1,$$

where $\Sp(2n,\F_2)$ is the symplectic group of the $\F_2$-vector space $\F_2^{2n}$.

For an odd integer $m\geq 3$, let $E_{m,n-1}^z$ be the group generated by
$u_1,\ldots,u_{n-1}$ with relations:
\begin{eqnarray}
 u_i^m&=&1\label{es1}\\
u_iu_{i+1}&=&z u_{i+1}u_i\label{es2}\\
u_iu_j&=&u_ju_i, \quad |i-j|>1\label{es3},
\end{eqnarray}
where $z$ is a central element of order $m$.

When $m$ is an odd prime $p$ and $n>1$ is odd, then $E_{p,n-1}^z$ is an extra-special $p$-group of exponent $p$.  In physics, $E_{p,n-1}^z$ is often called the finite Heisenberg group.  Then the outer automorphism group $Out_I(E_{p,n-1}^z)$ of $E_{p,n-1}^z$ (the subscript $_I$ of $Out_I$ means that we consider only automorphisms that are the identity on the center of $E_{p,n-1}^z$) is $\Sp(n-1, \F_p)$ \cite{Winter}.   We have another split short exact sequence:

$$1\rightarrow E_{p,n-1}^z \rightarrow \cC_{p,n-1} \rightarrow \Sp(n-1,\F_p) \rightarrow 1.$$

Therefore, the extra-special $p$-group $E_{p,n-1}^z$ can be regarded as the $p$-generalization of the Pauli group $\cP_{\frac{n-1}{2}}$, while the semi-direct product of the extra-special $p$-group by the symplectic group $\Sp(n-1, \F_p)$, denoted as $\cC_{p,n-1}$,  can be regarded as the generalization of the Clifford group $\cC_{\frac{n-1}{2}}$.

A similar interpretation can be made for $n$ even.  In this case, the center of $E_{p,n-1}^z$ is $\Z_p\oplus \Z_p$ rather than $\Z_p$ (the extra $\Z_p$ is generated by $u_1u_3\cdots u_{n-1}$.)  We call this group an almost extra-special $p$-group, though this terminology is not standard.  Below we will see that those groups naturally arise as images of braid group representations.

\subsubsection{Gaussian representations}

Let $\om$ be a primitive $m$-th root of unity and
consider the $\C$-group algebras $ES(\om,n-1)=\C[E_{m,n-1}^z]$, where $z=\om^{-2}$.
 Notice that $ES(\om,n-1)$ has dimension $m^{n-1}$ and is semisimple.  Therefore, $ES(\om,n-1)\cong \oplus_i M_{n_i}(\C)$, where $M_{n_i}(\C)$ is the full $n_i\times n_i$ matrix algebra and $\sum_i n_i^2=m^{n-1}$.

The Gaussian representation \cite{jones89} $\gamma:\B_n\rightarrow ES(\om,n-1)$ is defined on
braid generators of $\B_{n}$ by
$$\gamma(\sigma_i)=\frac{1}{\sqrt{m}} \sum_{j=0}^{m-1}\om^{j^2}u_i^j.$$

Direct computation shows that this leads to a unitary representation of the braid group. To match the representation $\rho_{X_\e}$ of $SO(m)_2$ exactly, we need to add a phase factor $e^{-\frac{\pi i(m^2-m)}{8}}$ to $\gamma(\sigma_i)$.  Since this phase factor is not important for our discussion below, we will omit it for convenience.

There is another representation of $\B_n$ in $ES(\om,n-1)$:

$$ \rho(\sigma_i)=(\frac{t+1}{m}\sum_{j=0}^{m-1}u_i^j)-1,$$
where $t+t^{-1}+2=m$.
This representation is usually called the Potts representation.  While for $m=3$, the Potts and Gaussian representations coincide, they  differ for $m\geq 5$.

\subsubsection{Metaplectic representations}

It was shown in \cite{GJo}
that the images of the Gaussian braid group representations are finite groups.
In fact, for $n$ odd the analysis in \cite{GJo} shows that, projectively,
$\gamma(\B_n)$ is isomorphic to the finite simple group
$\PSp(n-1,\mathbb{F}_p)$ when $m$ is a prime $p$.  More carefully, we have:

\begin{thm}

When $m$ is a prime $p$, the images ${\gamma}(\B_n)$ of the Gaussian representations are:
\begin{enumerate}

\item when $n$ is odd, then ${\gamma}(\B_n)/Z({\gamma}(\B_n))\cong \Sp(n-1, \bF_p)$, where the center $Z({\gamma}(\B_n))\cong \Z_4$ if $p=3$ mod $4$, and $\Z_2\oplus \Z_2$ if $p=1$ mod $4$.

\item when $n$ is even, then  ${\gamma}(\B_n)$ fits into the following exact sequence:

$$ 1\rightarrow E_{p,n}^z \rightarrow {\gamma}(\B_n) \rightarrow \Sp(n-2, \bF_p)\rightarrow 1.$$

\end{enumerate}

\end{thm}

For the braid group representation associated to the simple objects $X_\e$ of $SO(m)_2$, let $H_{n,X_\e}$ denote the images $\rho_{X_\e,n}(\B_n)$ of the $\B_n$ from the object $X_\e$.  Then for $m$ a prime $p$, $H_{n,X_\e}\cong {\gamma}(\B_n)$ for all $p$ \cite{RWen}.

A uniform way to understand the images ${\gamma}(\B_n)$ is to treat the short exact sequence for $n$ even as the definition of the symplectic group $\Sp(n-1, \bF_p)$.  The image of the braid group in the Gaussian representation
can be understood as follows. We can verify that:
\begin{eqnarray}\label{Conj-ui}
\left[\gamma(\sigma_{i+1})\right]^\dagger {u_i}\, \gamma(\sigma_{i+1}) &=& \omega^{-1} u_{i+1} u_i\cr
\left[\gamma(\sigma_{i-1})\right]^\dagger {u_i}\, \gamma(\sigma_{i-1}) &=& \omega u^{-1}_{i-1} u_i\cr
\left[\gamma(\sigma_{i})\right]^\dagger {u_i}\, \gamma(\sigma_{i}) &=& u_i\cr
\left[\gamma(\sigma_{j})\right]^\dagger {u_i}\, \gamma(\sigma_{j}) &=& u_i \, , \,\,|i-j|>1
\end{eqnarray}
Braiding transformations
are, therefore, automorphisms of $E_{p,n}^z/Z(E_{p,n}^z)$.
Hence, the image of the braid group is a subgroup
of the group of automorphisms of $E_{p,n}^z/Z(E_{p,n}^z)$.
This is equal to the {\it metaplectic representation} \cite{GJo} of
$Sp(n-1,\mathbb{F}_p)$.

\subsection{Finiteness for the qubit representation $\rho_{R_{Y_1}}$}\label{Qudit}

Let $R_{Y_1}$ be the following $8\times 8$ matrix, which is the block sum of two $4\times 4$ matrices:

$$\begin{pmatrix}
\nu \cosm &0&\im \sinm &0\\
0&-\im \sinm &0&\cosm \\
\im \sinm &0&\nu \cosm &0\\
0&\cosm &0&-\im \sinm
\end{pmatrix} \oplus
\begin{pmatrix}
-\im \sinm &0&\cosm&0\\
0&\nu \cosm &0&\im \sinm \\
\cosm &0&-\im \sinm &0\\
0&\im \sinm &0&\nu \cosm
\end{pmatrix},$$
where $\nu=-1$ if $m=3$ and $\nu=1$ if $m\geq 5$.

For the braid generator $\sigma_i\in \B_n$, set
$$\rho_{R_{Y_1}}({\sigma_i})=I^{\ot{(i-1)}} \otimes R_{Y} \otimes I^{\ot{(n-i-1)}},$$
then $\rho_{R_{Y_1}}$ is a representation of $\B_n$ on the $(n+1)$-qubit $(\C^2)^{\ot (n+1)}$.  We will call this representation the {\it qubit representation}.  This representation is a {\it generalized} localization of the $\rho_{Y_1}$ for $Y_1$ of the $JK_{r=6}$ theory and its generalizations to $m\geq 5$ \cite{RW}.  We do not know if there are localizations of $\rho_{Y_1}$'s that are not generalized.

By the computational basis of $n$-qudits $(\C^m)^{\ot n}$, we mean the basis consists of the tensor products $\{|j_1\rangle \ot |j_2\rangle \ot \cdots \ot |j_n\rangle, 0\leq j_1,j_2,\cdots, j_n\leq m-1\}$ of the standard basis $\{|j\rangle,j=0,..,m-1\}$ of $\C^m$ in Dirac notation.  When $m=2$, we have a one-to-one correspondence between the computational basis of $n$-qubits and the $n$-bit strings $|x_1\cdots x_n\rangle, x_i\in \{0,1\}$.  We will also denote the action of the Pauli matrices $\sigma_x$ and $\sigma_z$ on the $i$-th qubit as $X_i$ and $Z_i$, respectively, i.e. $X_i$ and $Z_i$ is a tensor product of the Pauli matrices $\sigma_x$ and $\sigma_z$ on the $i$th factor with the identity on the other tensor factors.

Let $\Lambda^2_{\rm{XOR}}\rm{NOT}$ be the XOR-controlled $3$-qubit gate defined on basis $|x_1 x_2 x_3\rangle$:

${\xornot}(|x_1 x_2 x_3\rangle)=|x_1 x_2 x_3\rangle$ if ${\rm{XOR}}(x_1,x_3)=0$ and

${\xornot}(|x_1 x_2 x_3\rangle)=|x_1 \rm{NOT}(x_2) x_3\rangle$ if ${\rm{XOR}}(x_1,x_3)=1$,

where ${\rm{XOR}}(x_1,x_3)=x_1+x_3 \; mod \; 2, \; {\rm{and}}\; {\rm{NOT}}(x_i)=1-x_i.$

Set $U_{i-1,i,i+1}=\rho_{R_{Y}}({\sigma_i})$, $H_i=Z_{i-1}X_iZ_{i+1}$ if $m\geq 5$ and $H_i=X_i$ if $m=3$, $V_i=e^{\frac{\pi \im }{m} H_i}$, and $\NOTi= \xornot_{i-1,i,i+1}$.

The index $i\in \{1,\cdots, n+1\}$ as there are $(n+1)$-qubits.  The qubit encoding
uses $X_\e$ (or $X_\e'$) particles at the two ends of the trivalent tree, and $(n+1)$ $Y_1$-particles in the middle as depicted here.  We show the case where both ends are $X_\e$:

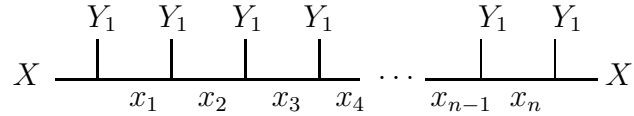
\vspace{.2in}
\begin{center}
\begin{picture}(250,30)
\label{pic:qubit-encoding}
\put(2,8){$ X$}
\put(30,30){$Y_1$}
\put(57,30){$Y_1$}
\put(85,30){$Y_1$}
\put(112,30){$Y_1$}
\put(46,0){$x_1$}
\put(72,0){$x_2$}
\put(100,0){$x_3$}
\put(124,0){$x_4$}
\put(160,0){$x_{n-1}$}
\put(190,0){$x_{n}$}
\put(18,10){\line(1,0){115}}
\put(34,10){\line(0,1){15}}
\put(62,10){\line(0,1){15}}
\put(90,10){\line(0,1){15}}
\put(118,10){\line(0,1){15}}
\put(140,10){$\dots$}
\put(158,10){\line(1,0){65}}
\put(179,10){\line(0,1){15}}
\put(207,10){\line(0,1){15}}
\put(178,30){$Y_1$}
\put(206,30){$Y_1$}
\put(226,8){$ X$}
\end{picture}
\end{center}
where ${x_i}=X_\e$ or $X_\e'$.
We set $Z_0=Z_{n+2}=+1$, and note that $\text{NOT}_i^{-+} \equiv I$ if $Z_{i-1} Z_{i+1}=1$ and
$\text{NOT}_i^{-+} \equiv X_i$ if $Z_{i-1} Z_{i+1}=-1$.
When the proof of a following lemma is a direct computation, we will simply omit it.

\begin{lemma}
\begin{equation}
U_{i-1,i,i+1}=V_i \cdot \NOTi,
\end{equation}
\end{lemma}

\begin{lemma}\label{Clifford}

The group generated by $\NOTi$ is a subgroup of the Clifford group, which will be denoted as $H$.

\end{lemma}

\begin{lemma}\label{NOT-Conj}
\begin{eqnarray}
\Bigl(\NOT_i\Bigr)^\dagger H_i \NOT_i &=& H_i, \\ \nonumber
\Bigl(\NOT_i\Bigr)^\dagger H_{i+1} \NOT_i &=& H_i H_{i+1}, \\ \nonumber
\Bigl(\NOT_i\Bigr)^\dagger H_{i-1} \NOT_i &=& H_{i-1} H_i, \\ \nonumber
\end{eqnarray}
and
\begin{equation}
|i-j|>1 \; \rightarrow \; \Bigl(\NOT_i\Bigr)^\dagger H_j \NOT_i = H_j.
\end{equation}
\end{lemma}

\begin{definition}
Define for $k\leq l$
\begin{equation}
S_{k,l}=\prod_{k \leq j \leq l} H_j.
\end{equation}
\end{definition}

\begin{lemma}
\begin{eqnarray}
i=k-1 & \quad \rightarrow \quad & \Bigl(\NOT_i\Bigr)^\dagger S_{k,l} \NOT_i=S_{k-1,l},
\\ \nonumber
i=k \; {\rm and} \; k<l  & \quad \rightarrow \quad & \Bigl(\NOT_i\Bigr)^\dagger S_{k,l} \NOT_i=S_{k+1,l},
\\ \nonumber
i=l \; {\rm and} \; k<l  & \quad \rightarrow \quad & \Bigl(\NOT_i\Bigr)^\dagger S_{k,l} \NOT_i=S_{k,l-1}
\\ \nonumber
i=l+1 & \quad \rightarrow \quad & \Bigl(\NOT_i\Bigr)^\dagger S_{k,l} \NOT_i=S_{k,l+1},
\\ \nonumber
{\rm Otherwise} &\quad \rightarrow \quad &
\Bigl(\NOT_i\Bigr)^\dagger S_{k,l} \NOT_i = S_{k,l}.
\end{eqnarray}
\end{lemma}

\begin{thm}

For each $n\geq 2$, the image $\rho_{R_{Y_1}}(\B_n)$ of the $n$-strand braid group $\B_n$ is a finite group $G\rtimes H$, where $G\cong \Z_2\times \Z_m^{n+1}$ and $H$ is the subgroup in Lemma \ref{Clifford}.

\end{thm}

\begin{proof}

The image group $\rho_{R_{Y_1}}(\B_n)$ is a subgroup of the group generated by the generators $\NOTi$ and $V_i$.
Note that
\begin{equation}
[V_i,V_j]=0
\end{equation}
for all $i,j$ and similarly
\begin{equation}
[H_i,H_j]=0
\end{equation}
for all $i,j$.
Consider a word $W$ in this (possibly larger) group consisting of a product of generators $\NOTi$ and $V_i$.    Our strategy is to commute the $V_i$'s to the right until the word is brought into a form of a product of $\NOTi$ generators followed by a product of $V_i$ generators (all $V_i$ generators appear on the right).

Note that $\NOTi$ is in the Clifford group.  So, it conjugates $Z_i X_{i+1} Z_{i+2}$ to another product of Pauli gates.
The matrix $\NOTi$ commutes with $H_i, H_{i+2}$, and $H_{i-2}$.  Also note that $\NOTi$ trivially commutes with $H_j$ for $|i-j|>2$.  By Lemma \ref{NOT-Conj}, $\NOTi$ conjugates a product of $H_j$ to some other product of $H_k$.
It follows that $\NOTi$ conjugates an exponential of a product of $H_j$ to an exponential of a product of $H_k$.
Therefore, any word $W$ can be written as a product of $\NOTi$ followed by
an exponential of a sum of terms, each term being of the form $\im l \frac{\pi}{m}$ multiplied by a product of $H_j$'s, where $l$ is some integer.

The group generated
by the operators $V_i$'s is an abelian group, which we
call $G$. Since $V_i^{2m}=1$, we can write an arbitrary element
of the group as $\prod_j V_j^{k_j}$, where the $k_j\in \{0, ..., 2m - 1\}$, so the group is a subgroup of $\Z^n_{2
m}$. However,
since $V_i^m= -1$, there are only $2 m^n$ distinct group
elements which can be written as $(\pm 1 ) \prod_j V_j^{k_j}$, where
the $k_j\in \{0, ...,m- 1\}$. This group is in
fact $\Z_2\times \Z^n_{m}$, and the generators of the group can be taken
to be $-V_i$ and $-1$. The group generated by the operators $\NOTi$ 
is a subgroup of the Clifford group; call this group
$H$. Then, because conjugation by $\NOTi$
defines an automorphism
of $G$, the group generated by $V_i$ and $\NOTi$ is the
semi-direct product $G \rtimes H$.

\end{proof}

The finiteness of braid group images for the representations $\rho_{Y_i}$ also can be deduced from \cite{NR}.  But the result in \cite{NR} does not give information on what the group is.  However, even our approach here does not give the {\it complete} information on the group.

\section{Generalized Knill-Gottesmann Theorems}

In this section, classical simulation of quantum circuits always refers to {\it efficient} classical simulation.

\subsection{Classical simulation of braid group representations}

Efficient classical simulation of Clifford circuits is provided by the Knill-Gottesman theorem (see e.g. \cite{NC}).  In the context of anyonic quantum computation, it implies that braiding quantum circuits based on the Ising anyon $\sigma$ can be efficiently simulated classically.

Fix a qudit $\C^m, m\geq 3$, the state space for $n$-qudits is $(\C^m)^{\ot n}$.  A quantum circuit model consists of a few fixed unitary matrices $\{g_i\}$, called {\it a gate set}, and all quantum circuits on $n$-qudits based on this gate set for all possible $n\geq 1$.  A {\it quantum circuit} on $n$-qudit is a composition of finitely many gates in the sense that gates should always be extended by tensoring an appropriate identity matrix if necessary.

Given a $R$-matrix $R_X: \C^m \ot \C^m \longrightarrow \C^m \ot \C^m$, by assigning $$I^{\otimes (i-1)}\otimes R_X \otimes I^{\otimes (n-i-1)}$$ to the braid generator $\sigma_i, i=1,\cdots, n-1$, we obtain a braid group representation $\rho_{R_X}$ of $\B_n$.  Naturally, we may regard this as a quantum circuit model with the gate set $\{\rho_{R_X}(\sigma_i^{\pm})\},i=1,2,\cdots, n-1.$  Then for each braid $\sigma \in \B_n$, $\rho_{R_X}(\sigma)$ is a braiding quantum circuit on $n$-qudits $(\C^m)^{\ot n}$ for all $n\geq 2$.  Note there are no $1$-qudit gates.

\begin{conj}
If a unitary (generalized) $R$-matrix $R_X$ is of finite order, then braiding quantum circuits can always be simulated classically.
\end{conj}
In this section, we will prove that this indeed is the case for the localizing qudit representation ${\rho}_{R_{X_\epsilon}}$, and also the generalized localizing qubit representation $\rho_{R_{Y_1}}$.

\subsubsection{Ising anyon $\sigma$}

The Ising anyon $\sigma$ can be localized by the following $R$-matrix \cite{FRW}, which consists of the Bell states.

$$R=\frac{1}{\sqrt{2}}\begin{pmatrix} 1&0&0& 1\\
 0 & 1&-1&0\\
 0&1&1 &0\\
 -1&0&0&1
 \end{pmatrix}.$$

As a corollary of the Knill-Gottesman theorem, any braiding quantum circuits from this $R$-matrix can be simulated classically.

\subsubsection{Localizing the Gaussian representation}

We want to find a unitary $m^2\times m^2$ matrix $U$ so that $U_i:=Id^{\ot i-1}\ot
U\ot Id^{\ot n-i-1}$ satisfy (\ref{es1}-\ref{es3}).  For this let
$\{\bfe_i\}_{i=1}^m$ denote the standard basis for $\C^m$ and define
\begin{equation}\label{Udef}
U(\bfe_i\ot\bfe_j)=\om^{i-j}\bfe_{i+1}\ot\bfe_{j+1}
\end{equation}
where the indices on $\bfe_{i}$ are to be taken modulo $m$.
It is straightforward to check that $U^m=I$ and $U^\dagger=U^{-1}$.  It is clear that $U_i$ and $U_j$ commute if $|i-j|>1$.
It remains to check (\ref{es2}).  For this it is enough to
consider $i=1$:
\begin{eqnarray*}
 &&U_1U_2(\bfe_i\ot\bfe_j\ot\bfe_k)=\om^{i-k-1}\bfe_{i+1}\ot\bfe_{j+2}\ot\bfe_{k+1}=\\
&&\om^{-2}\om^{i-k+1}\bfe_{i+1}\ot\bfe_{j+2}\ot\bfe_{k+1}=\om^{-2}
U_2U_1(\bfe_i\ot\bfe_j\ot\bfe_k)
\end{eqnarray*}
so $u_i\rightarrow U_i$ does give a representation of $ES(\om,n-1)$.  A standard
trace argument shows that this representation is faithful.  Thus defining
$$R_{X_\e}=\frac{1}{\sqrt{m}}\sum_{j=0}^{m-1}\om^{j^2}U^j$$ gives an $R_{X_\e}$-matrix localizing the
Gaussian representation, therefore the braid group representation $\rho_{X_\e,n}$ \cite{RW, RWen}.  The resulting braid group representation is the {\it qudit representation} $\rho_{X_\e}$, denoted as $\rho_{R_{X_\e}}$.

\subsubsection{Qudit representation $\rho_{R_{X_\e}}$}

For the qudit space $\C^m$, we denote its standard basis as $\{|j\rangle, 0 \leq j \leq  m-1\}$ as before.
On the $i$-th qudit $\C^m$, we define the \lq\lq shift" and \lq\lq clock" operators $X_i$ and $Z_i$.  The \lq\lq shift operator" $X_i$ is the permutation matrix $X_i |j\rangle=|j+1\rangle$, and the \lq\lq clock operator" $Z_i$ is the diagonal matrix $Z_i|j\rangle=\omega^{j}|j\rangle, 0 \leq j \leq  m-1$.  When $m=2$, $X_i, Z_i$ are the Pauli matrices.

The qudit representation is a direct sum of many sectors with multiplicities.  Each sector has a graphical calculus using the following tree, where 
the two horizontal ends of the trivalent tree are labeled by two anyons $Y_L, Y_R$ in $\{Y_i, i=0,1,...,r\}, m=2r+1$:

\vspace{.2in}
\begin{center}
\begin{picture}(250,30)
\label{pic:qubit-encoding}
\put(2,8){$ Y_L$}
\put(30,30){$X_\e$}
\put(57,30){$X_\e$}
\put(85,30){$X_\e$}
\put(112,30){$X_\e$}
\put(46,0){$x_1$}
\put(72,0){$x_2$}
\put(100,0){$x_3$}
\put(124,0){$x_4$}
\put(160,0){$x_{n-1}$}
\put(190,0){$x_{n}$}
\put(18,10){\line(1,0){115}}
\put(34,10){\line(0,1){15}}
\put(62,10){\line(0,1){15}}
\put(90,10){\line(0,1){15}}
\put(118,10){\line(0,1){15}}
\put(140,10){$\dots$}
\put(158,10){\line(1,0){65}}
\put(179,10){\line(0,1){15}}
\put(207,10){\line(0,1){15}}
\put(178,30){$X_\e$}
\put(206,30){$X_\e$}
\put(226,8){$Y_R$}
\end{picture}
\end{center}

Now we define stabilizer formalism, the meaning of classically efficiently simulable in the sense of the stabilizer formalism, and what do we mean by stabilizer measurements are classically efficiently simulable in the stabilizer formalism.

\begin{definition}\label{stabilizerformalism}

Given a sequence of Hilbert spaces ${\mathcal H}$ which are $N$-fold tensor products of a qudit $\C^D$.
\begin{itemize}
\item[1.] A {\it stabilizer group} is a subgroup of the unitary group $U(D)$ whose order is at most $O(1)$ in $N$. Usually, the order of the stabilizer group is a constant such as $D^2$ in the Pauli matrices for $D=2$.

\item[2.] A {\it stabilizer S} is an operator acting on ${\mathcal H}$ which is the tensor product of $N$ operators chosen from the stabilizer group.  The operators in this product may be distinct.

\item[3.] Given a finite set of pairs $(S_i,\omega_i)$, where $S_i$ is a stabilizer and $\omega_i$ a complex number.  
A nonzero vector $\phi$ is {\it stabilized} by that set if $S_i \phi = \omega_i \phi$ for all pairs $(S_i,\omega_i)$, i.e. $\phi$ is a common eigenvector of all stabilizer operators with the specified eigenvalues.

\item[3.] A {\it complete set of stabilizers} is a set of $N$ pairs $(S,\omega)$ such that
there is a unique vector (unique up to phase and overall normalization) $\phi$ such that $\phi$ is stabilized by that set.
Such a vector $\phi$ is called a {\it stabilizer state}.
\end{itemize}

Given a set $M$ of unitary matrices (unitaries) on each ${\mathcal H}$, the sequence of unitaries $M$ is said to be {\it classically efficiently simulable in the sense of the stabilizer formalism} if 

\begin{itemize}
\item[1.] There is a  stabilizer group such that
given any stabilizer $S$ and given any unitary $U\in M$, the operator $U S U^\dagger$ is a stabilizer.

\item[2.]  Given a stabilizer $S$ presented as a list of $N$ elements of the stabilizer group,
it is possible in polynomial time on a classical computer to compute $U S U^\dagger$ in the same presentation.
\end{itemize}
Note that item 1 means that, given any set of pairs $(S,\omega)$, and any vector $\phi$ stabilized by that set, then the set of pairs
$(U S U^\dagger,\omega)$ stabilizes the state $U \phi$.

We say that {\it stabilizer measurements are classically efficiently simulable in the stabilizer formalism} if 
given any complete set of stabilizers and given any other stabilizer $T$
\begin{itemize}
\item[1.] It is possible in probabilistic polynomial time on a classical computer to output a complex number $z$ such that the probability that $z=\omega$ is equal
\begin{equation}
\langle \phi, P(T;\omega) \phi \rangle,
\end{equation}
where $\phi$ is a state stabilized by that complete set of stabilizers with $|\phi|=1$, and
$P(T;\omega)$ denotes the projection onto the eigenspace of $T$ with eigenvalue $\omega$.

\item[2.] After outputing a given $z$, it is possible in polynomial time on a classical computer to output another complete 
set of stabilizers which stabilizes the state $P(T;\omega) \phi$.
\end{itemize}

\end{definition}

\begin{thm}

 When $m$ is a prime $p$, then all polynomial length braiding circuits of $X_\e$ anyons can be efficiently simulated classically in the sense of definition \ref{stabilizerformalism}.

\end{thm}

\begin{proof}

There are $(r+1)^2$ different sectors.  Since braiding $X_\e$ anyons will not mix sectors, so they can be simulated simultaneously or individually.  The representation space of $\rho_{X_\e,n}$ is $\textrm{Hom}(Y_L\ot (X_\e)^{\ot n}, Y_R)$ with a basis consisting of eigenstates of 
products of $u_i$'s.  They are labeled trees in Fig. \ref{pic:qubit-encoding}.    
By equation (\ref{Conj-ui}), braiding $X_\e$ anyons
transforms products of $u_i$'s into products of $u_i$'s. As a result, the evolution of a state in $\textrm{Hom}(Y_L\ot (X_\e)^{\ot n}, Y_R)$ can be efficiently simulated classically by following
the evolution of these operators. 

The extra-special $p$-group $E^z_{p,2k}$ has another generating set \cite{HNW}:
\begin{eqnarray}
\label{eqn:Heisenberg-Z-X}
X_{i} X_{j} &=& X_{j} X_{i} \, ,\,\,  Z_{i} Z_{j} = Z_{j} Z_{i}\cr
X_{i} Z_{j} &=& \, z^{\delta_{ij}}\,Z_{j} X_{i}\cr
X_{i} z &=& z X_{i} \, ,\,\,  Z_{i} z = z Z_{i}
\end{eqnarray}
These two presentations of  $E^z_{p,2k}$ are related
by $u_{2i-1}=X_{i}$, $u_{2i}=Z_{i}Z^\dagger_{i+1}$ for $i\neq k$
and $u_{2k}=Z_{k}$.  We will include $E^z_{p,2k-1}$ inside $E^z_{p,2k}$.

The extra-special $p$-group $E^z_{p,2k}$ introduces redundant states for $X_\e$-anyons, which will be removed by the stabilizer formalism.  
Braidings commute with roughly half of the generators
of $E^z_{p,2k}$.   Note that $U_1, . . .U_{n-1}, \tilde{U}_1, . . . \tilde{U}_{n-1},X_1Z_1,X_nZ^{\dagger}_n, z$,
where $U_i=X_iX_{i+1}Z_iZ^{\dagger}_{i+1}, \tilde{U}_i = X_iX_{i+1}Z^{\dagger}_i�Z_{i+1}$ is another generating set of $E^z_{p,2k}$. The generators $\tilde{U}_i, X_1Z_1$, and
$X_nZ^{\dagger}_n$ all commute with the $U_i$'s and, therefore, with braiding.

It suffices to show that the two-qubit gates $R_{X_\e}$ conjugates each of the generators $X_1, X_2, Z_1, Z_2$ to a product of these generators, up to a phase.  A direct computation shows: 
\begin{equation}
U=X_1 X_2 Z_1 {Z_2}^{\dagger},
\end{equation}
then $U$ commutes with the operators $X_1 X_2$, $Z_1 Z_2^\dagger$, $X_1 Z_1$, and $X_2 Z_2^\dagger$ as may be checked.
So, all these $4$ operators are mapped to a product of the generators.
Note that these $4$ operators are not all independent: the product of the first three operators is equal to the fourth operator, up to a phase.

Note that the operators $Z_1,X_1X_2,Z_1 Z_2^\dagger,X_1 Z_1$ generate the image group.
Hence, it suffices to check that $Z_1$ is mapped to a product of generators by $R_{X_\e}$, up to phase.
Let $l=\frac{p+1}{2}$.
Note that
\begin{equation}
U^j Z_1 = \omega^{-j} Z_1 U^j.
\end{equation}

Then we have:
\begin{eqnarray}
&& R_{X_\e} Z_1 R_{X_\e}^\dagger  \\ \nonumber
&=& \xi \sum_{j=0}^{p-1} \omega^{j^2} U^j Z_1 R_{X_\e}^\dagger \\ \nonumber
&=& Z_1 \xi \sum_{j=0}^{p-1} \omega^{j^2-j} U^j R_{X_\e}^\dagger \\ \nonumber
&=& Z_1 \xi \sum_{j=0}^{p-1} \omega^{(j-l)^2-l^2} U^j R_{X_\e}^\dagger  \\ \nonumber
&=& Z_1 \xi \omega^{-l^2} \sum_{j=0}^{p-1} \omega^{(j-l)^2} U^j R_{X_\e}^\dagger  \\ \nonumber
&=& Z_1 \xi \omega^{-l^2} \sum_{j=0}^{p-1} \omega^{j^2} U^{j+l} R_{X_\e}^\dagger \\ \nonumber
&=& Z_1 \xi \omega^{-l^2} U^l \sum_{j=0}^{p-1} \omega^{j^2} U^j R_{X_\e}^\dagger \\ \nonumber
&=& Z_1 \omega^{-l^2} U^l R_{X_\e} R_{X_\e}^\dagger \\ \nonumber
&=& Z_1 \omega^{-l^2} U^l,
\end{eqnarray}
where $\xi=\frac{1}{\sqrt{m}}$.

Therefore, the evolution of $Z_1$ can be
efficiently simulated classically and, consequently, so can
the evolution of any state stabilized by products of $X_i$ and $Z_j$
operators. 
\end{proof}

For the standard quantum circuit model, measurements in the middle of computation can be all postponed to the end (see e.g. \cite{NC}).  But this is not the case for the anyonic computational model.  So we may ask if we can also simulate some measurements in the middle of a computation.  Measurements in a basis of products of $X_i$ and $Z_j$ operators can be simulated classically as in the Clifford circuit case. 
Thus, we conclude that we can efficiently simulate
classically any polynomial length quantum operation that consists of creating pairs of $X_\e$ 
anyons out of the vacuum, braiding them, and then measuring
them in a basis of products of $X_i$ and $Z_j$ operators.

More difficult measurements are projections of a pair of $X_\e$ anyons onto a definite charge.  While projecting onto the trivial charge can be simulated classically, we do not know if this is true for nontrivial charges.  This question seems to be open even in the Ising case.

When $m$ is not a prime, then the equivalence of the Gaussian representation and the braid group representation $\rho_{X_\e}$ is conjectured to be true, but not known.  Note that the exact same proof gives a classical simulation of the Gaussian representation.   

Instead of $u_i^{m}=1$, if we set $u_i^{2m}=1$, then the Gaussian representation and its localization are defined similarly.  The same proof above will give a classical simulation of the localized braid representation.

\subsection{Qubit representation $\rho_{R_{Y_1}}$}

If $R_Y: V^{\ot 3}\rightarrow V^{\ot 3}, V=\C^m$ is a solution to a $(2,3,1)$-gYBE, for the braid generator $\sigma_i\in \B_n$, by setting
$$R_{\sigma_i}=I^{\ot{(i-1)}} \otimes R \otimes I^{\ot{(n-i-1)}}, $$
we again have a braid group representation \cite{KW}.  But in this case, $\B_n$ acts on the vector space $V^{\ot (n+1)}$.
The above discussion about braiding quantum circuits applies to the generalized localizations as well.

We will use the same notation as in Section \ref{Qudit}.

\begin{thm}
There exists a description of elements of $G \rtimes H$ which uses only polynomial space on a classical computer, and there exists a classical algorithm to multiply elements of $G \rtimes H$ written in this description which uses only polynomial time. Consequently, the image of any polynomial length braid in $\rho_{Y_1}$ can be computed in this classical description in classical polynomial time.
\end{thm}

\begin{proof}

As explained above, we can simulate each sector independently.  
 In Section \ref{Qudit}, the image group is given as a subgroup of a semi-direct product of two groups $G\rtimes H$.  $G$ is an abelian group with generators $e^{i \frac{\pi i}{m} S_{k,l}}$ for each pair of integers $k,l$ with $1\leq k \leq l\leq n+1$.  Thus, the group $G$ has $O(n^2)$ generators and is finite order so we can write each group element by writing a sequence of $n^2$ integers in the range $\{0,...,2m-1\}$.

The other group $H$ is generated by the $\NOTi$ operators.  This is a subgroup of the Clifford group.  For a unitary $U$ in the Clifford group, let us compute $U X_i U^\dagger\; \textrm{and} \; U Z_i U^\dagger$ for all $i$.
If one knows these conjugations of Pauli operators, then one can compute $U O U^\dagger$ for any operator $O$ since any $O$ can be written as sums of products of the $X_i$ and $Z_i$, where $X_i$ and $Z_i$ are Pauli operators.  Since a unitary in the Clifford group conjugates Pauli operators to products of Pauli operators, we can store $U X_i U^\dagger$ and $U Z_i U^\dagger$ for all $i$ with polynomial resources.
This is how we will describe elements of the two groups $G$ and $H$.
Note that the set of all $U X_i U^\dagger$ and $U Z_i U^\dagger$ uniquely specifies $U$ up to an overall phase.  In the present case, the phase is uniquely determined by the fact that $U$ has all matrix elements real and non-negative in the computational basis.

The classical simulation is done as follows.  After some number of steps, we have a unitary $U$ stored as $AC, A\in G, C\in H$, where we store $A$ by storing the $n^2$ integers and we store $C$ by its action on the Pauli operators.  We want to left-multiply $U$ by a generator which is either a generator in the Abelian group $G$ or an element of the group $H$.
Note that the braid group representation image is generated by a product of an abelian group element $A'$ and a $C'$ operator, but we describe the left-multiplication by the abelian elements $A'$ and the $C'$ operators separately.  Left-multiplying by an abelian generator $A'$ is easy: simply multiply $A$ by the abelian generator.  Left multiplying by a $C'$  is done in two steps by considering the identity $C'AC=(C'A{C'}^\dagger) C'C=A'C'C$.  First, we commute the $C'$ through the element $A$.  Using the commutation relations we have given, we can do this multiplication in polynomial time: each $S_{k,l}$ gets conjugated by the $C'$ to some other $S_{k,l}$, so commuting the $C'$ through $A$ permutes the different integers describing $A$.  Next, we left-multiply $C$ by the $C'$; we can do this by taking $C X_i C^\dagger$ and conjugating by the $C'$ which can be done in polynomial time, and similarly for $C Z_i C^\dagger$.

\end{proof}

Since any operator $S_{kl}$ is conjugated by this group to another $S_{kl}$, we are able to classically simulate braiding combined with measurement of the operators $S_{kl}$.  Physically, the measurement of $S_{i,i+1}$ corresponds to whether or not the $i$-th and $(i+1)$-th anyons fuse to a $Y_i$ anyon or to ${\bf{1}}$ or $Z$.  However, this 
measurement does not distinguish the fusion outcomes ${\bf{1}}$ or $Z$ from each other.  Ultimately, this is related to the $\#P$-hardness of evaluating the related link invariant in the next section.

\section{\#P-hardness of Evaluating Link Invariants}

In this section, we relate the evaluation of the link invariant $I_{Y_1}(L)$ for certain links to the computation of the Ising model partition function of some graphs $G$.  By our construction, the maximum cuts of the graphs $G$ correspond to ground states of the Ising model.  The evaluation of the link invariant $I_{Y_1}(L)$ would imply the counting of the maximum cuts of the graphs $G$, a well-known $\#P$-hard problem.

Suppose $\cI(L;x,y)$ is a polynomial invariant of oriented links and $p_0=(x_0,y_0)$ is an algebraic point that all evaluations $\cI(L;p_0)$ exist.  Then $\cI(L;p_0)$  is a numerical invariant of oriented links, which sometimes can be identified with a partition function of a $(2+1)$-TQFT.  We are  interested in the computational complexity of such evaluations $\cI(L;p_0)$ for all oriented links. To be precise, we need to specify an input encoding for links.  There are several equivalent encodings such as combinatorial data for link diagrams or words in braid generators. In Lemma \ref{partitionfn} below, we use the plat closure of braids.

There are two $2$-variable generalizations of the Jones polynomial: the HOMFLY polynomial $P(L;l,m)$, and the Kauffman polynomial $F(L;a,z)$.  The Jones polynomial $V(L;t)$ is the specialization of the Kauffman polynomial  $F(L;t^{-\frac{3}{4}}, -(t^{-\frac{1}{4}}+t^{\frac{1}{4}}))$.  For these two polynomials, the complexity of evaluating each algebraic point has been shown to be either $\# P$-hard or polynomially computable classically \cite{Welsh}.  When the computation is polynomial classically, the evaluation has an interpretation using classical topological invariants.  We will refer to the classically polynomial time computable points as the {\it classical points}.

\subsection{Evaluating $I_{X_\e}(L)$}

For the simple object $X_\e$ in $SO(m)_2$, the resulting link invariant $I_{X_\e}(L)$ has a classical interpretation.
They are not specializations of the HOMFLY or the Kauffman polynomials when $m\geq 7$, and satisfy a $(\frac{m+5}{2})$-term skein relation.
When $m=3$, $I_{X_\e}(L)$ is the Jones polynomial at a $6$th root of unity and when $m=5$, an evaluation of the Kauffman polynomial at a $5$th root of unity \cite{jones89}.

\begin{prop}

The exact evaluation of $I_{X_\e}(L)$ is polynomial time classically when $m$ is a prime $p$.

\end{prop}

\begin{proof}
Here we present links as link diagrams and measure complexity using the number of crossings.

In \cite{jones89,GJo}, $I_{X_\e}(L)$ is expressed as a sum using the symmetrized Seifert surface $S$ of a braid closure: $$(\frac{1}{p})^{\textrm{genus}}\sum_{v\in H_1(S;\F_p)} \omega^{<v,v>},$$ where $<,>$ is the Seifert form.  The norm of this sum is $(\sqrt{p})^r$, where $r$ is the rank of the first $mod$-$p$-homology of the $2$-fold branched cover of the $3$-sphere $S^3$ along $L$.  
This betti number can be computed from the symmetrized Seifert matrix efficiently.  The phase is given by the Legendre symbol $\Bigl(\frac{i^{\left \lfloor n/2 \right \rfloor}det(A)}{p}\Bigr)$, which is also efficiently computable.
\end{proof}

\subsection{Evaluating $I_{Y_1}(L)$}

For the simple object $Y_1$ of $JK_{r=6}$ and its generalizations, the resulting link invariant normalized to oriented links is essentially the evaluation of the Kauffman polynomial at $(a,z)=(-\im e^{-\frac{\pi \im}{m}}, 2sin (\frac{\pi}{m}))$.  Precisely,
 $$I_{Y_1}(L)=\frac{1}{2} (-1)^{c(L)-1} F(L; -\im e^{-\frac{\pi \im}{m}}, 2sin (\frac{\pi}{m}))$$
for a $c(L)$-component oriented link $L$ \cite{HongYBE}.  When $m=3$, $I_{Y_1}(L)$ is the evaluation of the colored Jones polynomial at a $6$th root of unity.

\begin{thm}

The exact evaluation of $I_{Y_1}(L)$ is $\# P$-hard.

\end{thm}

\begin{proof}

For the Kauffman polynomial $F(L;a,z)$ such that $a\neq 0, z\neq 0$, the classical points are \cite{Welsh}:

\begin{enumerate}

\item $a=\pm \im$; or

\item $(a,z)=(-q^{\pm 3}, q+q^{-1})$, where $q^{16}=1$ or $q^{24}$=1, but $q\neq \pm \im$; or

\item $(a,z)=(q^{\pm 3}, q+q^{-1})$, where $q^{8}=1$ or $q^{12}$=1, but $q\neq \pm \im$; or

\item $(a,z)=(-q^{\pm 1}, q+q^{-1})$, where $q^{16}=1$, but $q\neq \pm \im$; or

\item $(a,z)=\pm (1, q+q^{-1})$, where $q^{5}=1$.

\end{enumerate}

Noticing that $(a,z)=(-\im e^{-\frac{\pi \im}{m}}, 2sin (\frac{\pi}{m}))$ are not in the list, we obtain the desired result.

\end{proof}

This proof of the $\#P$-hardness of evaluating $I_{Y_1}(L)$ is very indirect.  To have a better understanding of the $\#P$-hardness, we provide two refinements.

\begin{thm}

\begin{enumerate}

\item Any sufficiently accurate approximation of the link invariants $I_{Y_1}(L)$ is $\# P$-hard.

\item The sign of $I_{Y_1}(L)$ is $\# P$-hard to approximate.

\end{enumerate}

\end{thm}

The specialization of the Kauffman polynomial $F(L;a,z)$ to $z=-(a+a^{-1})$ has a state sum due to R.\ Lickorish and K.\ Millet \cite{LM}:

$$E(L)={2}(-1)^{c(L)-1}F(L)=\sum_{S \subset L} a^{-4 \langle S,L-S\rangle},$$\label{L-M}

\noindent where the summation is over the $2^{c(L)}$ sublinks of $L$ including the empty link, $a=-i\exp(-i \pi/m)$, and $\langle S,L-S\rangle$ is the total linking number of the two sublinks $S$ and $L-S$.  This state sum of $I_{Y_1}(L)$ allows us to translate Ising model partition functions to $I_{Y_1}(L)$ for certain links.

\subsection{$\#P$-hardness of approximation}

We first define a partition function for an Ising spin system, a quantity frequently used in physics and closely related to the Tutte polynomial.
The restriction to integer entries for the matrix below is because those will be the only cases we need.
\begin{definition}
Let $J=(J_{ij})_{1\leq i,j\leq N}$ be a symmetric $N\times N$ matrix with integer entries and with diagonal entries of $J$ equal to $0$.  Let $y$ be any real number.
Define the {\it Ising partition function} to be
\begin{eqnarray}
\label{ZJdef}
Z(J,y)&=&\sum_{\sigma\in \{-1,1\}^N} y^{\sum_{i<j} J_{ij} \delta_{\sigma_i,\sigma_j}},
\end{eqnarray}
where the sum is over $N$-dimensional vectors $\sigma$ with entries chosen from the set $\{-1,1\}$ and
where $\delta_{\sigma_i,\sigma_j}$ is the Kronecker symbol.

The quantities $y^{\sum_{i<j} J_{ij} \delta_{\sigma_i,\sigma_j}}$ are called the {\it Boltzmann weights}.
\end{definition}

Often in physics one restricts to $y>0$ so that the Boltzmann weights are positive, but we will not make this restriction here.
Also, often in physics one considers a closely related quantity:
\be
\label{oftenphysics}
\sum_{\sigma\in \{-1,1\}^N}
\exp(-\beta \sum_{i<j} J_{ij} \sigma_i \sigma_j).
\ee
Setting
\be
\beta=-\frac{1}{2}\ln(y),
\ee
the quantity of Eq.~(\ref{oftenphysics}) is equal to
\begin{eqnarray}
\sum_{\sigma\in \{-1,1\}^N}
\exp(-\beta \sum_{i<j} J_{ij} \sigma_i \sigma_j)&=&
\sum_{\sigma\in \{-1,1\}^N}
\exp\Bigl(-\beta \sum_{i<j} J_{ij} (2\delta_{\sigma_i,\sigma_j}-1)\Bigr)\\ \nonumber
&=&Z(J,y) \exp(\beta \sum_{i<j} J_{ij})
\\ \nonumber
&=& Z(J,y) \Bigl( \sqrt{y} \Bigr)^{-\sum_{i<j} J_{ij}},
\end{eqnarray}
so that Eq.~(\ref{oftenphysics}) is the same as Eq.~(\ref{ZJdef}) up to a trivial multiplication by $\exp(\beta \sum_{i<j} J_{ij})$.  We prefer to use Eq.~(\ref{ZJdef}) because it will be simpler later when dealing with negative $y$.

Lemma \ref{partitionfn} below shows that for certain $J,y$, we can define a link whose link invariant is equal to the partition function $Z(J,y)$, up to multiplication by a quantity that can be computed easily.  The allowed $y$ take certain discrete values depending upon $m$.  The allowed $y$ obey $|y|\leq 1$, and we will choose $y$ such that $|y|<1$ to obtain a nontrivial $Z(J,y)$.  Importantly, the size of the link will only be polynomially large.

We then give two different proofs of $\#P$-hardness, each leading to a slightly different improvement of the $\#P$-hardness.   In the first proof, we choose $y$ and $J$ such that all the Boltzmann weights are positive; then, by taking the coupling constants in the Ising system large, a sufficiently accurate evaluation of the Ising partition function allows one to count the number of maximum cuts a given graph has, a known $\#P$-hard problem.
The second proof is based on considering a different regime using negative Boltzmann weights; in this case, a known result \cite{GoldbergJerrum} shows that evaluating the sign of $Z(J,y)$ is $\#P$-hard.

The sign of the Boltzmann weights depends upon the sign of $y$ and on the parity of the entries of $J$.
For every odd $m\geq 3$, there exists a choice of $\pow$, with $1 \leq \pow \leq m-1$, such that $y$ in Eq.~(\ref{ydef}) below is negative and greater than $-1$; this regime will give a complex $\beta$.  In this case, the Boltzmann weights may be negative if $J$ has odd entries and this case is the basis of the second proof.  To obtain positive Boltzmann weights, we can either restrict to matrices $J$ with even entries, or we can consider odd $m>3$, for which it is possible to pick $\pow$ such that $y$ is positive and less than $1$.
To keep a unified treatment for all $m$, we follow the choice of using matrices with even entries in the first proof.

The positive Boltzmann weight approach has the advantage that the quantity that we show is $\#P$-hard to approximate has a simpler experimental realization than the second regime does: the absolute value squared of this quantity is the probability that preparing pairs of $Y$ particles that fuse to the vacuum, then braiding them in a way determined by the link, and then fusing them together in pairs will have all pairs fuse to the vacuum, while the sign is not in principle measurable without doing an experiment that interferes different particle trajectories.
Also, the result below that it is $\#P$-hard to approximate this quantity to a multiplicative accuracy better than a given amount (this accuracy is $\exp(-{\textrm{poly}}(N))$ can be straightforwardly translated into a result that it is $\#P$-hard to approximate to an additive accuracy better than a given amount since we can give a lower bound on the quantity itself, being a sum of positive terms.

The negative Boltzmann weight approach of lemma \ref{lemmaneg} below has the advantage that since we show that it is $\#P$-hard to approximate the sign of the link invariant, it shows that approximating the link invariant to any (positive) multiplicative accuracy is hard, improving greatly on the exponentially small multiplicative accuracy required in the first regime.

\begin{lemma}
\label{partitionfn}
Let $J=(J_{ij})$ be a symmetric $N\times N$ matrix with $J_{ij}\in \Z, J_{ii}=0$, and $\pow$ be any integer.
Set
\be
\label{ydef}
y=\frac{a^{-4\pow}+a^{+4\pow}}{2},
\ee
where
\be
a=-i\exp(-i \pi/m).
\ee

Define $P(J)$ to be the sum of the positive entries of $J$, and $A(J)=\sum_{1\leq i,j\leq N} |J_{ij}|$.  Note that $P(J)$ and $A(J)$ both are even since $J$ is symmetric.

Then, there is a link $L$ such that the quantity $E(L)={2} (-1)^{c(L)-1} F(L)$ defined by Lickorish and Millett \cite{LM} obeys
\be
\label{claim}
E(L)=Z(J,y) a^{-2\pow P(J)}    \Bigl(\sqrt{y}\Bigr)^{-\sum_{i<j} J_{ij}}    \Bigl(\sqrt{2(a^{-4\pow}+a^{+4\pow})}\Bigr)^{A(J)/2},
\ee
where $c(L)$ is the number of components of the link and
where we choose the sign of the square-roots such that $\sqrt{y} \sqrt{2(a^{-4\pow}+a^{+4\pow})}=a^{-4\pow}+a^{+4\pow}$.  The number of components $c(L)$ is equal to
$N+\sum_{i<j} |A_{ij}|$.

The number of crossings in the link is at most polynomial in $N+\pow A(J)$.  The link can be presented as the plat closure of a braid of $2c(L)$ strands of length which is at most polynomial in $N+\pow A(J)$.
\begin{proof}
Lickorish and Millett show that $E(L)$ can be written as a sum
\be
\label{LMeq}
E(L)=\sum_{S \subset L} a^{-4 \langle S,L-S\rangle}.
\ee
If a link $L$ has $c(L)$ components, then a sublink $S$ can be specified by specifying, for each component, whether that component is in $S$ or not.
We do this by defining, for each sublink $S$, a vector $s$ with entries $s_i$ for $i=1,...,c(L)$, such that $s_i=+1$ if the $i$-th component is in $S$ and $s_i=-1$ otherwise.
Let $\langle i,j \rangle$ be the link number between the $i$-th component and the $j$-th component.
The invariant $\langle S,L-S\rangle$ is then equal to the sum of $\langle i,j\rangle$ over pairs $i\in S$ and $j\in L-S$.

So,
\be
-4 \langle S,L-S\rangle=-\sum_{i \neq j} (1+s_i) (1-s_j) \langle i,j \rangle=-2\sum_{i < j} (1-s_i s_j) \langle i,j \rangle
\ee
So, Eq.~(\ref{LMeq}) is equal to
\be
\label{boltzmann}
E(L)=\sum_{s \in \{-1,1\}^{c(L)}} a^{2 \sum_{i < j} s_i s_j \langle i,j \rangle}
a^{-2 \sum_{i < j} \langle i,j \rangle}.
\ee

We now show how to construct the desired link $L$.  We pick
\be
c(L)=N+\sum_{i<j} |J_{ij}|.
\ee
So, we can define a one-to-one mapping from the set of triples of integers, $i,j,n$
with $1\leq i<j\leq N$ and $1 \leq n \leq |J_{ij}|$ to the set of integers from $N+1$ to $c(L)$.  We specify this map by adding integers from $1$ to $|J_{ij}|$ to $N$ following the lexicographical order of $(i,j)$'s, and denote this mapping by a function $F(i,j,n)$.

Pick $\langle i,j \rangle=0$ if both $i$ and $j$ are less than or equal to $N$ or if both $i$ and $j$ are greater than or equal to $N$.
Consider each triple $i,j,n$ as above.  Then pick
\be
\langle i,F(i,j,n)\rangle=\pow
\ee
for all $n$.  Pick
\be
\langle j,F(i,j,n)\rangle=\pow\cdot {\rm sgn}(J_{ij})
\ee
 for all $n$, where ${\rm sgn}(x)$ is the sign function: $+1$ for $x>0$ and $-1$ for $x<0$.  Note that by the symmetry of the linking number $\langle  i,F(i,j,n)\rangle = \langle F(i,j,n),i\rangle$ and similarly $\langle j,F(i,j,n)\rangle=\langle F(i,j,n),j\rangle$.  Let all other link numbers $\langle k,l\rangle=0$.

This completes the description of the link; we now show that the link invariant $E(L)$ is the desired result.
We do this by defining
\be
X(L,\sigma)=\sum_{s\in \{-1,1\}^{c(L)}}^{s_i=\sigma_i, 1\leq i \leq N}
a^{2\sum_{i<j} s_i s_j \langle i,j \rangle}
a^{-2\sum_{i < j} \langle i,j \rangle}.
\ee
The summation notation means the sum over all $c(L)$-dimensional vectors $s$ with with entries $+1$ or $-1$, such that the first $N$ entries of $s$ are equal to those of $\sigma$ for some $N$-dimensional vector $\sigma$.
Then,
\be
E(L)=\sum_{\sigma \in \{-1,1\}^N} X(L,\sigma).
\ee

To compute $X(L,\sigma)$, note that this is equal to
$$X(L,\sigma)=a^{-2\sum_{1\leq i < j\leq c(L)} \langle i,j \rangle}$$
\begin{eqnarray}
&&
\prod_{1 \leq i<j \leq N}\,
\Bigl( \prod_{n=1}^{|J_{ij}|} \; \Bigl(\sum_{s_{F(i,j,n)}\in \{-1,1\}} \,a^{2s_i s_{F(i,j,n)} \langle i,F(i,j,n)\rangle
 +2s_j s_{F(i,j,n)} \langle j,F(i,j,n)\rangle} \Bigr)
 \Bigr) \\ \nonumber
&=& a^{-2\pow P(J)} \prod_{1 \leq i<j \leq n} \, W(s_i,s_j,J_{ij}),
\end{eqnarray}
where
$$W(s_i,s_j,J_{ij})=$$
\begin{eqnarray}
&&
\prod_{n=1}^{|J_{ij}|} \Bigl(  \sum_{s_{F(i,j,n) )\in \{-1,1\}} }} a^{2s_i s_{F(i,j,n)\langle i,F(i,j,n)\rangle
 +2s_j s_{F(i,j,n)} \langle j,F(i,j,n)\rangle} \Bigr)
\\ \nonumber
&=&
\Bigl(
 \sum_{s_{F(i,j,n)} \in \{-1,1\}} a^{2s_i s_{F(i,j,n)} \langle i,F(i,j,n)\rangle
 +2s_j s_{F(i,j,n)} \langle j,F(i,j,n)\rangle} \Bigr)^{|J_{ij}|}.
\end{eqnarray}

A direct computation gives in the case
 $s_i={\rm sgn}(J_{ij}) s_j$ that $W(s_i,s_j,J_{ij})=(a^{-4\pow}+a^{+4\pow})^{|J_{ij}|}$ and
in the case that $s_i=-{\rm sgn}(J_{ij})s_j$ that $W(s_i,s_j,J_{ij})=2^{|J_{ij}|}$.
Let $y$ be as defined in Eq.~(\ref{ydef}) and let
\be
z=2(a^{-4\pow}+a^{+4\pow}).
\ee
Then
\be
W(s_i,s_j,J_{ij})=\sqrt{y}^{J_{ij} s_i s_j} \sqrt{z}^{|J_{ij}|},
\ee
where the ambiguity in the sign of the square-root is resolved by choosing $\sqrt{y} \sqrt{z}=a^{-4\pow}+a^{+4\pow}$.

So,
\be
X(L,\sigma)= y^{\sum_{i<j} J_{ij} \delta_{\sigma_i,\sigma_j}}  \Bigl(\sqrt{y}\Bigr)^{-\sum_{i<j} J_{ij}}
a^{-2\pow P(J)} \sqrt{z}^{\sum_{ i<j} |J_{ij}|}.
\ee
Comparing to Eq.~(\ref{claim}) completes the proof.

To present the link as the plat closure of a braid, define a braid with $2c(L)$ strands.  Above, we have specified the linking number $\langle i,j \rangle$ for each pair of components $i,j$ of the link.  The braid will be the product of a sequence of shorter braids, one for each pair $i<j$ for which $\langle i,j \rangle \neq 0$.  In each such shorter braid, apply the product of braid group generators $\sigma_{2(j-1)} ... \sigma_{2i+1} \sigma_{2i}$  which moves the $2i$-th strand to the right, braiding it clockwise around neighboring strands, until it is immediately to the left of the $2j-1$-th strand.  Apply $\sigma_{2j-1}^{+d}$ or $\sigma_{2j-1}^{-d}$ depending on the desired sign of the linking number $\langle i,j \rangle$.  Then apply the inverse of the product
$\sigma_{2(j-1)} ... \sigma_{2i+1} \sigma_{2i}$.  This completes the description of the shorter braids.  The plat closure of the braid which is the product of these shorter braids is the desired link.  Since there are only polynomially many pairs $i,j$, the braid has polynomial length as claimed.
\end{proof}
\end{lemma}

\begin{definition}
Given a graph $G$, with a set of vertices $V$, define a cut of $G$ to be a partitioning of the vertices of the graph into two subsets, $S$ and $V-S$.  Define the size of a cut $S,V-S$ to be the number of edges in the graph connecting vertices in $S$ to vertices in $V-S$, and denote it as
$\rmsize(G,S)$.  Define $M(G)$ to be the maximum over all cuts of graph $G$.  Define $N(G)$ to be the number of cuts with size equal to $M(G)$, so that
\be
N(G)=| \{S\subseteq V| \rmsize(G,S)=M(G) \}|.
\ee

Note that according to this definition we consider $S,V-S$ to be a distinct cut from $V-S,S$, so that a cut is defined by an ordered pair of sets.
\end{definition}

\begin{lemma}
\label{lemmagraphapprox}
Consider any $N\times N$ matrix $J$ with all diagonal entries equal to zero and with all other entries equal to either $0$ or $K$ for some even integer $K>0$.

Define a graph $G$ to have $N$ vertices, labelled by integers between $1$ and $N$, and to have edges connecting vertices $i,j$ if and only if $J_{ij}\neq 0$.

Fix $\pow$ such that Eq.~(\ref{ydef}) gives $|y|<1$.

Then, the quantity $Z(J,y)$ defined in Eq.~(\ref{ZJdef}) obeys
\be
\label{approx}
N(G) |y|^{K (|E|-M(G))} \leq Z(J,y)
 \leq
N(G) |y|^{K (|E|-M(G))}+  |y|^{K (1+|E|-M(G))} 2^{N},
\ee
where $|E|$ is the cardinality of the edge-set of the graph so that $K |E|=\sum_{ij} J_{ij}$.
\begin{proof}
There is an obvious one-to-one mapping between choices of the vector $\sigma$ in Eq.~(\ref{ZJdef}) and cuts: for a given choice of the vector $\sigma$, define a set $S$ to contain all vertices $i$ such that $\sigma_i=+1$ and define $V-S$ to contain all other vertices.
Then, Eq.~(\ref{ZJdef}) gives $Z(J,y)$ as a sum over cuts.  For a maximum cut, the Boltzmann weight equals $|y|^{K (|E|-M(G))}$,
while for all other cuts the Boltzmann weight is at least a factor of $|y|$ smaller.  Since there are only $2^{N}$ cuts and since every cut contributes a positive term in the sum, the desired result follows.
\end{proof}
\end{lemma}

\begin{lemma}
\label{lemmapos}
Fix $\pow$ such that Eq.~(\ref{ydef}) gives $|y|<1$.

We consider symmetric matrices $J$ such that all diagonal entries are equal to zero and all other entries are equal to either $0$ or $K$ for
$K>0$ chosen to be the smallest even integer greater than or equal to
\be
(N+1) \ln(2)/\ln(|y|).
\ee

Consider the problem of approximating $Z(J,y)$ to a multiplicative accuracy within the range $[1-2^{-N-3},1+2^{-N-3}]$.  That is, we want to compute a number $\tilde Z$ such that it is guaranteed that $(1-2^{-N-3})\tilde Z \leq Z(J,y) \leq (1+2^{-N-3}) \tilde Z$.

This approximation problem is $\#P$-hard.  Further, since $\pow$  can be chosen independently of $N$, $K$ is bounded by a constant times $N$ and so
lemma \ref{partitionfn} constructs a link $L$ with a number of crossings that is polynomial in $N$ such that the evaluation of the link invariant $E(L)$ computes the partition function $Z(J,y)$ using Eq.~(\ref{claim}).
\begin{proof}
Consider an arbitrary graph $G$ with $N$ vertices. Let the matrix $J$ be $K$ times the adjacency matrix of $G$.  We will show that being able to evaluate $Z(J,y)$ to the given accuracy allows one to compute $M(G)$ and $N(G)$.

Note that $N(G)\leq 2^N$.
Also note that for the given $K$,
$ |y|^{K (1+|E|-M(G))} 2^N \leq (1/2) |y|^{K (|E|-M(G))}$.  Therefore, all graphs $G$ with a given $M(G)$ have a $Z(J,y)$ that is at least twice as large as that of any graph $G'$ with $M(G')<M(G)$.
So, given an approximation $\tilde Z$ with the given accuracy, it is possible to determine $M(G)$ exactly since $M(G)$ is an integer.

Next we show that we can determine $N(G)$.
For the given $K$, Eq.~(\ref{approx}) gives
\begin{eqnarray}
N(G) |y|^{K (|E|-M(G))} \leq Z(J,y) & \leq &
N(G) |y|^{K (|E|-M(G))}+  |y|^{K (1+|E|-M(G))} 2^{N}
\nonumber
\\ \nonumber
& \leq &  (N(G)+1/2)|y|^{K (|E|-M(G))} .
\end{eqnarray}
Equivalently,
\be
N(G) \leq \frac{Z(J,y)}{ |y|^{K (|E|-M(G))}} \leq N(G)+1/2.
\ee

Let
\be
\tilde N=\frac{\tilde Z}{ |y|^{K (|E|-M(G))}}.
\ee
So, $(1-2^{-N-3})\tilde N\leq N(G)+1/2$ so
\be
(1-2^{-N-3})\tilde N -1/2
\leq N(G) \leq (1+2^{-N-3})\tilde N.
\ee
Since $N(G)\leq 2^N$, it follows that $\tilde N\leq (2^{N}+1/2) (1+2^{-N-3}) <2^{N+1}$
So, $2^{-N-3}\tilde N<1/4$.  So, since $N(G)$ is an integer, the above equation determines $N(G)$.

So, given this approximation, it is possible to determine both $M(G)$ and $N(G)$.  That is, it is possible to determine both the MAX-CUT of a graph (an $NP$-hard problem) and to count the number of such maximal cuts, a $\#P$-hard problem \cite{sharpMAXCUT}.
\end{proof}
\end{lemma}

\begin{lemma}
\label{lemmaneg}
Fix $\pow$ such that Eq.~(\ref{ydef}) gives $-1<y<0$.
Then, the quantity $Z(J,y)$ is real but it is $\#P$-hard to compute its sign, even in the case that the entries of $J$ are either $0$ or $1$.

 Further, since $\pow$  can be chosen independently of $N$, if the entries of $J$ are either $0$ or $1$ (hence bounded in magnitude),
lemma \ref{partitionfn} constructs a link $L$ with a number of crossings that is polynomial in $N$ such that the evaluation of the link invariant $E(L)$ computes the partition function $Z(J,y)$ using Eq.~(\ref{claim}).
\begin{proof}
If the entries of $J$ are either $0$ or $1$, a computation of the sign of $Z(J,y)$ gives a computation of the sign of the Tutte polynomial $T_G(x,y)$ of a graph $G$ whose adjacency matrix is $J$ with $y$ as given here and with $x$ such that $(x-1)(y-1)=2$.  However, computing this sign is $\#P$-hard \cite{GoldbergJerrum}.
\end{proof}
\end{lemma}

Our improvement for $\#P$-hardness is to relate the link invariant $I_{Y_1}(L)$ to the Ising model partition function.  For certain links, the link invariants allow one to approximate the Ising partition function.  We have seen that approximating the Ising model partition function has at least two  different regimes, one regime where it is $\#P$-hard and one regime where it is $NP$-hard, depending upon whether we compute both $N(G)$ and $M(G)$ or just $M(G)$.  Approximating  $I_{Y_1}(L)$ is more general than just approximating the Ising model partition function at real temperature because the corresponding temperature could be complex, i.e., by picking some very specific links, we constructed a problem with a real temperature, but for arbitrary links we have to deal with a complex temperature.  However, if we assume that all we are doing is approximating the Ising model partition function when we approximate the link invariant, the problem goes from $NP$ to $\#P$ as we approximate it more accurately, while link invariants of computationally universal theories like $SU(2)_3$ instead go from $BQP$ to $\#P$, which is quite different.  Since $BQP$ is believed to be incomparable to $NP$, there is no approximation regime in which such other link invariants would be $NP$.
It is not clear if this $NP$-hard regime will survive complex temperature or not, though, but it would be a very interesting phenomenon.

In \cite{KR},  the quantum double of finite groups are studied.  Some theories could also have finite braid group images but $\#P$-hard link invariants.  It is likely that our theory follows the same pattern
of their analysis, where different approximations to the link invariant
are in $P$, or are $SBP$-hard, or are $\#P$-hard, depending
upon the accuracy of the approximation. It would be interesting
to see if quantum doubles can be efficiently simulated classically.

\end{document}